\documentclass{article}
\usepackage[utf8]{inputenc}
\usepackage[dvips]{graphicx}
 \usepackage{amssymb,amsmath}
 \usepackage{amsthm}
\usepackage{epic}
\usepackage{pstricks}
\usepackage{pst-node}
\usepackage{color}

\usepackage{tikz}
\usetikzlibrary{arrows,decorations.markings}
\usepackage[utf8]{inputenc}
\usepackage{pstricks-add}

\newtheorem{theorem}{Theorem}
\newtheorem{proposition}[theorem]{Proposition}

\newtheorem{corollary}[theorem]{Corollary}
\newtheorem{definition}[theorem]{Definition}
\newtheorem{remark}[theorem]{Remark}

\newcommand{\G}{\Gamma }

\date{}

\begin{document}


\title{A little more about bipartite biregular cages, block designs and generalized polygons}
\author{Gabriela Araujo-Pardo\footnote{Research supported in part by PAPIIT-UNAM-M{\' e}xico IN101821 and Proyecto de Intercambio Académico de la Coordinación de la Investigación Científica de la UNAM "Bipartite biregular cages, block designs and generalized quadrangles".
} \\
garaujo@im.unam.mx
\\
\\
Gy\"orgy Kiss\footnote{This research was supported in part  by the Hungarian National Research, Development and Innovation Office  OTKA grant no. SNN 132625.} \\
gyorgy.kiss@ttk.elte.hu
\\
\\
Tamás Sz\H onyi\footnote{This research was supported in part by the Hungarian National Research, Development and Innovation Office  OTKA grant no. K 124950.}  \\
tamas.szonyi@ttk.elte.hu}
\maketitle

\begin{abstract}
In this paper, we obtain new lower and upper bounds for the problem of bipartite biregular cages. Moreover, for girth $6$, we give the exact parameters of the 
$(m,n;6)$-bipartite biregular cages when $n\equiv -1$ $\pmod m$ using the existence of Steiner System system $S(2,k=m,v=1+n(m-1)+m)$. 

For girth $g=2r$ and $r=\{4,6,8\}$, we use results on $t$-good structures given by ovoids, spreads and sub-polygons in generalized polygons to obtain 
$(m,n;2r)$-bipartite biregular graphs. We emphasize that, as we improve the lower bounds on the order of these graphs, we also prove that some of them are 
$(m,n;2r)$-bipartite biregular cages.

In particular, we construct relatively small bipartite biregular graphs from a special class of generalized quadrangles and hexagons. In a special case, we show that the graph obtained is actually a $(3,4;8)$-bipartite biregular cage on $56$ vertices.
\end{abstract}

\section{Introduction, definitions and basic properties}

The cage problem is a classical problem in extremal graph theory. In a regular graph of degree $k$ and girth $g$ we ask for the minimum number of vertices. Since it is a classical problem and there is a dynamic survey by Exoo and Jajcay \cite{ExooJaj08} about it, we do not go in detail. It is natural to extend the problem for non-regular graphs. In particular, Yuansheng and Liang \cite{YuLi03} extended it to biregular graphs. Later Filipowski, Jajcay and Ramos-Rivera \cite{FilRamRivJaj19} further specialized the problem for bipartite biregular graphs, that is for bipartite graphs with the property that the vertices of one bipartition class have degree $m$ and vertices of the other bipartition class have degree $n$. In this variant of the problem, the girth is clearly even, and we may assume that $g>4$. Such a bipartite biregular graph is the incidence graph of a regular uniform set system, so one may expect that cages will correspond to special set systems. This connection was made clear in the paper by Araujo-Pardo, Jajcay, Ramos-Rivera, and Sz\H onyi \cite{AJRS22}. The aim of the present paper is to add further remarks to the paper \cite{AJRS22}. Actually, bipartite biregular graphs and the above variant of the cage problem were already introduced by F\"uredi, Lazebnik, Seress, Ustimenko, and Woldar \cite{FuLaSeUsWol95}, who considered the problem in the asymptotic sense. They called a bipartite graph an $(r,s,t)$-graph if it is bipartite biregular, where the two degrees are $r$ and $s$ and the girth is $2t$. Among other things, they showed that for every $r,s,t\ge 2$ there exist such graphs and proved several asymptotic results for them. 
In the past few years, many authors have studied bipartite cages and obtained different results regarding their size, see \cite{AraExoJaj, BobJajPis, ExooJaj16, Fil17, JWGJS}.

In Section 2, we consider girth $2r$ with $r$ odd and are mainly interested in improving slightly the lower bounds given in the paper \cite{AJRS22}, using divisibility conditions for bipartite biregular graphs. 

In the last section, we consider girth $g=2r$ with $r\geq 4$ even and improve the general lower bounds given previously also in \cite{AJRS22}. In particular, for $r=\{4,6,8\}$, we divide the section into two parts. First, we construct 
$(m,n;2r)$-bipartite biregular graphs by deleting vertices and edges of the incidence graph of some specific generalized polygons. In the second part, we use $t$-good and 
$(s,t)$-good structures, in particular ovoids, spreads and sub-polygons in generalized polygons to obtain new $(m,n;2r)$-bipartite biregular graphs. The notion of $t$-good structure was introduced by G\'acs and H\'eger \cite{GH,DHS} while the $(s,t)$-good structures first appeared in a paper of Beukemann and Metsch \cite{BeuM}. Previously, these structures were only used to construct regular graphs, we now extend them to investigate bipartite biregular graphs. We emphasize that, as we  improve the lower bounds on the order of these graphs, we also prove that some of them have the minimum order, 
consequently they are 
$(m,n;2r)$-bipartite biregular cages. In particular, we construct relatively small bipartite biregular graphs from a special class of generalized quadrangles and hexagons. For example, we show that one of the graph obtained is actually a 
$(3,4;8)$-bipartite biregular cage on $56$ vertices.

Here we give only the most necessary definitions. For a detailed introduction to block designs and Steiner systems, we refer the reader to \cite{cvl} and \cite{ColDin}, while the concepts from finite geometries we use can be found for example in \cite{KSz} and \cite{VM98}.

Let $\G$ denote a finite, connected, simple graph with vertex set $V=V(\G)$ and edge set $E=E(\G)$.
Let $d$ denote the minimal path-length distance function of $\G$. 
For $v \in V$ and an integer $i$ we let $\G_i(v) = \{w \in V \mid d(v,w)=i\}$.  For an edge $uv$ of $\G ,$ let $D_j^i(u,v)=\G_i(u)\cap \G_j(v).$ If $\G$ is not a tree, then the {\em girth} $g$ of $\G$ is the length of the shortest cycle in $\G$. If $C$ is a cycle of $\G$ of girth length $g$, 
then we refer to $C$ as a {\em girth cycle} of $\G$. 

\begin{definition}
Let $2\leq m\leq n$ and $4\leq g$ be integer numbers. An {\emph{$(m,n;g)$-bipartite biregular graph}} is a bipartite graph of even girth $g$ having degree set equal to $\{ m,n\} $ and satisfying the additional property that the vertices in the same bipartition class have the same degree.  We denote it as 
$(m,n;g)$-bb-graph and call it, by simplicity, a bb-graph. An $(m,n;g)$-{\emph{bipartite biregular cage}} (shortly $(m,n;g)$-bb-cage) is an $(m,n;g)$-bb-graph of minimum order, we also call it a bb-cage. Moreover, $B_c(m,n;g)$ denotes the size of an 
$(m,n;g)$-bb-cage.
\end{definition}

Let $\G $ be a $(2,n;g)$-bb-graph with bipartition $A,B$ such that vertices in 
$B$ have degree two. Then we can define a graph $\G '$ in the following way: 
$V(\G ')=A$ and there is an edge between vertices $u$ and $v$ if and only if $d(u,v)=2$ in $\G .$ Then $\G '$ is an $(n,g/2)$-graph and $\G $ is an $(m,n;g)$-bb-cage if and
only if $\G '$ is an $(n,g/2)$-cage.
Therefore, in the rest of this paper we also assume $m >2$.

In \cite{FilRamRivJaj19}, Filipovski, Ramos-Rivera and Jajcay proved the generalized Moore Bound $B(m,n;2r)$. 


\begin{proposition}[Bipartite Moore Bound for even $r$]
Let $ 3 \le m < n $ and $r$ even, then:
\begin{equation}
\label{genMoore-even}
B(m,n;2r)=
(m+n)\frac{\left( (m-1)(n-1)\right) ^{\frac{r}{2}}-1}{(m-1)(n-1)-1}.
\end{equation}
\end{proposition}

\begin{definition} 
\label{excess}
If $\Gamma$ is a graph that attains this lower bound we call it a 
\emph {bb-Moore cage}. Then, as usual, the difference between the order of an 
$(m,n;2r)$-graph $\Gamma $ and the corresponding bipartite Moore bound $B(m,n;2r)$ is called the 
\emph{excess} of $\Gamma $. Consequently, a graph of excess $0$ is a bb-Moore cage.
\end{definition}

It is well-known that, if $m\neq n$, this bound is only attained when $r$ is even, and the corresponding graphs are related to the incidence graphs of certain finite geometries.

 The \emph{incidence graph} (also known as \emph{Levi graph}) of a point-line
incidence geometry is a bipartite graph whose bipartition sets correspond to the
set of points and lines, respectively, and there is an edge between two vertices if
and only if the corresponding point is incident with the corresponding line.

The next "folklore" statement gives an important correspondence between generalized polygons and biregular graphs. The proof can be found for example in  
\cite[Lemma 1.3.6]{VM98}.
\begin{theorem}
\label{g-g}
A finite thick generalized $r$-gon $\mathcal{G}$ exists if and only if there exists a connected bipartite biregular graph $\G$ of diameter $r$ and girth $2r,$ such that each vertex has degree at least three. In this case, $\G$ is the incidence graph of $\mathcal{G}$.
\end{theorem} 

Two easy corollaries of this theorem were given by Araujo-Pardo et al. \cite{AJRS22}.
We recall (and reformulate) their results only in the $r$ even case. 

\begin{corollary}[{Araujo-Pardo et al., \cite[Theorem 4.]{AJRS22}}]
\label{cor-3}
If $r$ is even and $3\leq m\leq n$, then a bipartite biregular $(m,n;2r)$-graph of order $B(m,n,2r)$ is necessarily the Levi graph of a generalized polygon, and $r=4,$ $6$
or $8.$ Conversely, the Levi graph of a generalized $r$-gon for $r=4,\, 6,\, 8$ is a
bb-Moore cage.  
\end{corollary}

\begin{corollary}[{Araujo-Pardo et al., \cite[Corollary 8.]{AJRS22}}]
\label{cor-8}
Let $3\leq m\leq n,$ $r\geq 4$, and suppose that no generalized 
$r$-gon of order $(m-1,n-1)$ exists. Then
$$B_c(m,n;2r)\geq B(m,n;2r)+\frac{m+n}{\mathrm{gcd}(m,n)}.$$
\end{corollary}

Actually, in the paper \cite{FilRamRivJaj19}, the authors also define the excess when $r$ is
odd. However, if $r$ is odd, the bound given by counting the vertices of the Moore tree is only attained when $m=n$ (the regular case). The following bounds were given in \cite{AJRS22} and improve the obvious Moore bound showing that the excess cannot be small.

\begin{proposition}[Bipartite Bound for odd $r$]
Let $ 3 \le m < n $ and $r$ even, then:
\begin{equation} \label{gen-lower-bounds-odd}
 B_c(m,n,2r)\geq
 \begin{array}{c}1 + n + n(m-1) + n(m-1)(n-1) + n(m-1)^2(n-1) + \ldots  \\ + \; n(m-1)^{\frac{r-1}{2}}(n-1)^{\frac{r-3}{2}} + 
\left \lceil \frac{n(m-1)^{\frac{r-1}{2}}(n-1)^{\frac{r-1}{2}}}{m} \right \rceil. \end{array}
\end{equation}
Moreover, if $ \frac{n(m-1)^{\frac{r-1}{2}}(n-1)^{\frac{r-1}{2}}}{m}  $ is not an integer, then we have that: 
\begin{equation} \label{special-lower-bounds-odd} 
B_c(m,n,2r)\geq
\begin{array}{c} 1 + n + n(m-1) + n(m-1)(n-1) + n(m-1)^2(n-1) + \ldots  \\ + \; n(m-1)^{\frac{r-1}{2}}(n-1)^{\frac{r-3}{2}} + 
\left \lceil \frac{n(m-1)^{\frac{r-1}{2}}(n-1)^{\frac{r-1}{2}}}{m} \right \rceil + \left \lceil \frac{n}{m} \right \rceil. \end{array}
\end{equation}
\end{proposition}

\section{Lower bounds and $bb$-cages of girth $6$ arising from Steiner Systems}
\label{girth6}
In this section, we generalize the results given in Theorems 16 and 17 of \cite{AJRS22} for $m=3$ and $n\equiv 2 \pmod 3$. We prove a general result for $n > m$ and 
$n\equiv -1 \pmod m$ using elementary divisibility conditions coming from counting edges in a bipartite graph.

\begin{theorem}
Let $3\leq m\leq n$ and $n\equiv -1 \pmod m$. Then the size of the $(m,n;6)$-bipartite biregular cage is at least 
$(\frac{n}{m}+1)(n+1)(m-1)$ and one can obtain a cage of this size by deleting a point and all the blocks through the deleted point, from a Steiner system $S(2,k=m,v=1+n(m-1)+m)$ (if it exists).
\end{theorem}

\begin{proof} Let $\Gamma $ be our cage and denote the two classes by $S$ and $L$. Start from a vertex $s$ of the smaller bipartition class $S$. Then $s$ has $n$ neighbours in $L$ and there are $n(m-1)$ vertices in $S$ that are at distance 2 from $s$. Hence $|S|\ge 1+n(m-1)$. Observe that $(1+n(m-1))n$ is not divisible by $m$, hence $|S|=1+n(m-1)+x$ for some $x>0$. As $|S|n$ is the number of edges in $\Gamma$, it can also be expressed as $|L|m$, so $(1+n(m-1)+x)n\equiv 0 \pmod m$. This implies $(1+(-1)(m-1)+x)\cdot (-1)\equiv 0 \pmod m$ and so $x \equiv m-2 \pmod m$. The smallest such $x$ is $m-2$, showing that $|S|\ge 1+n(m-1) +m-2= (n+1)(m-1)$. The assertion on the size of $\Gamma $ follows immediately by counting the number of edges. In case of equality, the graph has $|S|=(n+1)(m-1)$ vertices of degree $n$ and $|L|=\frac{n(n+1)(m-1)}{m}$ vertices of degree $m$. 

For the construction, let us consider the vertices of $S$ as the points of an incidence structure and add a new point. The total number of points will be $v+1=1+(n+1)(m-1)$, and $m-1$ divides $v-1$. As $(n+1)\equiv 0$ $\pmod m$, we have that $m$ divides $n+1$, and consequently $m$ divides $v$. Then,  $m(m-1)$ divides $v(v-1)$, so the divisibility conditions for the existence of an $S(2,m,v+1)$ are satisfied. 

If there is a Steiner system with these parameters, then we have $v+1=1+(n+1)(m-1)$ elements and by the rules of the parameters on Steiner Systems, we have that the number of blocks in the design is $\frac{[(n+1)(m-1)+1](n+1)}{m}$ and any element is in $n+1$ blocks. Deleting a point and all the blocks through it, then the number of points will be $v$ and any point is in $n$ blocks (the blocks joining the actual point with the deleted one will be missing). So, the incidence graph of this structure will indeed have the required parameters: 

It has $(n+1)(m-1)$ points and $\frac{[(n+1)(m-1)+1](n+1)}{m}-(n+1)=\frac{n(n+1)(m-1)}{m}$ blocks. Then, the incidence graph is a $(m,n;6)$-bipartite biregular graph of order $(\frac{n}{m}+1)(n+1)(m-1)$, and it attains the previous lower bound. 
Note also that when $n$ is large enough compared to $m$, then the results of Wilson
\cite{Wil} guarantee the existence of such a design.  
\end{proof}

In the simplest case $m=3$ (for which the above results were shown in \cite{AJRS22}), we can say a little more. Actually, in this case, the cage mentioned in the above theorem must come from a Steiner triple system. 

\begin{theorem}
Let $n\equiv 2 \pmod 3$. Then the $(3,n;6)$-cage has size $2+2n+(2+2n)n/3$ and it has to be the incidence graph of a Steiner triple system on $2n+3$ points with one point and all the triples through it deleted.
\end{theorem}
Before the proof let us remark that $n=3k+2$ implies that $v+1=2n+3=6k+4+3$ and there is a Steiner system 
on $v+1$ points, since it is congruent to 1 modulo 6.
\begin{proof}
We follow the proof of the previous theorem. It implies that the bipartition class $S$ has $2n+2$ points.
For each point $s$ in $S$, there is a unique point $s'$ which is at distance greater 
than $2$ from $s$. We also see that $s$ and every vertex of $S$ different from $s'$ 
will have a unique common neighbour in $L$.  Again, let us call the vertices of $S$ points, the vertices of $L$ blocks, and add a new point $\infty$ to $S$. The points $s$ and $s'$ will form a block together with $\infty$. This means that if we add the new point and these blocks, then we get a Steiner triple system. The original graph is the incidence graph of the structure obtained by removing $\infty$ and the blocks containing it.
\end{proof}

Notice that this does not mean that the graph is unique since there are many pairwise non-isomorphic Steiner systems but the construction scheme is the same (just the STS we start from can be different).

We can also use the above proofs to give a slightly more explicit formulation of Theorem 17 in \cite{AJRS22} in the case $g=2r$, $r$ odd.

\begin{theorem}
Let $g=2r$, $r$ odd, and let $2<m<n$ be the given degrees in a bipartite biregular graph. Let $x$ be the smallest non-negative integer for which $(1+n(m-1)+\ldots +n(m-1)^{r-2}(n-1)^{r-3}+x)\cdot n$ is divisible
by $m$. Then any bipartite, biregular cage with degrees $m$ and $n$ contains at least 
$((1+n(m-1)+\ldots +n(m-1)^{r-2}(n-1)^{r-3}+x))\cdot (1+n/m)$ vertices.
\end{theorem}

\begin{proof}
Consider the bipartite biregular cage $\Gamma$ with degrees $m$ and $n$.
Start from a vertex $v$ of the smaller bipartition class. This class will contain all those vertices which are at an even distance less than $r$ from $v$. The number of such vertices is $(1+n(m-1)+\ldots +n(m-1)^{r-2}(n-1)^{r-3})$, using the fact that the vertices of the smaller bipartition class have degree $n$, and those of the larger class have degree $m$. Hence the number of vertices in the smaller bipartition class is $(1+n(m-1)+\ldots +n(m-1)^{r-2}(n-1)^{r-3}+x)$ for some non-negative integer $x$. Counting the edges gives the divisibility condition and the lower bound in the theorem immediately.
\end{proof}

\section{Families of {\emph{bb-cages}} of girth $g=2r$ with even $r\geq 4$}\label{girth8}
Throughout this section, we assume that $r\geq 4$ even. 
We generalize some constructions given in \cite{AJRS22} for girth $2r.$ In particular, 
we present new families of graphs of girths $8$ and $12$ with small excess.

We divide this section into two parts depending of the technique that we use in each of them. 
\subsection{Constructions by deleting vertices and edges of the incidence graphs of generalized polygons.}

The following description is also given in \cite{AJRS22} and \cite{FilRamRivJaj19}:

The existence of bipartite biregular $(m,n;2r)$-graphs, with $r\geq 4$ even, in relation 
to the lower bound  (\ref{genMoore-even}) on $B(m,n;2r)$ is based on the idea of the existence of a tree, denoted by ${\cal{T}}_{uv}$. 
It has $B(m,n;2r)$ vertiices and is called therein the {\em{Moore tree}}. 
It is necessarily contained in any bipartite biregular $(m,n;2r)$-graph
and consists of trees ${\cal{T}}_u$ and ${\cal{T}}_v$ attached to the vertices $u$ 
and $v$ of an edge $e= \{ u,v \} $. Each of the two trees consists of levels of vertices of varying degrees $n$ and $m$ and is of depth $r-1$ (one starting of a root of degree $n-1$ and one starting of a root of degree $m-1$), and the last level consisting of vertices of degree $1$. 
Thus, ${\cal{T}}_{uv}$ is the union of the edge $e=\{u,v\}$ and two disjoints trees ${\cal{T}}_u$ and ${\cal{T}}_v$, rooted in $e=\{u,v\}$, where $V({\cal{T}}_{uv})=V({\cal{T}}_u)\cup V({\cal{T}}_v)$ and $E({\cal{T}}_{uv})=E({\cal{T}}_u)\cup E({\cal{T}}_v)\cup \{ e \} $. A bipartite biregular $(m,n;2r)$-cage of order 
$B(m,n;2r)$ exists if and only if
the tree ${\cal{T}}_{uv}$ can be completed into an $(m,n;2r)$-graph by adding 
$n-1$ edges to each leaf of ${\cal{T}}_u$ connecting these leaves 
to the leaves of the tree ${\cal{T}}_v$ in such a way that each leave of 
${\cal{T}}_v$ ends up being of degree $m$, and no cycles shorter than $2r$ are introduced.

Using techniques introduced in the study of the $(k;g)$-cage problem (see for example \cite{AGMS07,ABGMV08,ABH10}), we obtain the following two theorems. 
\begin{theorem}\label{generalizationtoptree}
Suppose that there exists a generalized $r$-gon $\mathcal{G}$ of order $(m,n)$. 
Let $m_1\leq n_1$ be integers satisfying the inequalities $m_1\leq m,$ $n_1\leq n$ 
and 
$$(m_1+n_1)n^{\frac{r}{2}-1}m^{\frac{r}{2}-1}<B(m_1,n_1;2(r+1)).$$
Then 
there exist $(m_1,n_1;2r)$-bb graphs of order
$(m_1+n_1)n^{\frac{r}{2}-1}m^{\frac{r}{2}-1}.$
\end{theorem}

\begin{proof}
Let $\Gamma$ be the Levi graph of $\mathcal{G}$. Then, by Corollary \ref{cor-3}, 
$\Gamma $ is an $(m+1,n+1;2r)$-bb Moore cage. Let $uv$ be any edge of $\Gamma $. As argued above, $\Gamma$ contains the Moore tree ${\cal{T}}_{uv}$ with the additional edges connecting the leaves of the two subtrees 
${\cal{T}}_u$ and ${\cal{T}}_v$. We may assume that $\mathrm{deg}(v)=m$ and $\mathrm{deg}(u)=n.$
  Let $D_1(v)=\{ v_1,v_2,\dots ,v_m\}$ and  $D_1(u)=\{ u_1,u_2,\dots ,u_n\}.$ 
 Notice that each vertex in $D_{r-1}^r(v,u)$ has exactly one neighbour in each $D_{r-2}^{r-1}(v,v_i)$, for $1\leq i \leq m$, and analogously, each vertex in $D_{r-1}^r(u,v)$ has exactly one neighbour in each $D_{r-2}^{r-1}(u,u_j)$ for for $1\leq j \leq n$. 
 
Moreover, there exists a matching between any two sets of vertices of  $D_{r-2}^{r-1}(v,v_i)$ and  $D_{r-2}^{r-1}(u,u_j)$, for all 
$1\leq i \leq m$ and $1\leq j \leq n$. 
  
Let $m_1\leq m$, and $n_1\leq n$ and take $G_{m_1,n_1}$ the induced graph of the set of vertices $D_{r-2}^{r-1}(v,v_i)\cup D_{r-2}^{r-1}(u,u_j)$, for all 
  $1\leq i \leq m_1$ and $1\leq j \leq n_1$.
  Then $G_{m_1,n_1}$ is a bipartite biregular graph with degrees $m_1$ and $n_1$ 
and its order is
  $$m_1n^{\frac{r}{2}-1}m^{\frac{r}{2}-1}+n_1n^{\frac{r}{2}-1}m^{\frac{r}{2}-1}=(m_1+n_1)n^{\frac{r}{2}-1}m^{\frac{r}{2}-1}.$$
  
Clearly, the girth does not decrease, because we only delete vertices and edges. On the other hand, the order of $G_{m_1,n_1}$ is less than $B(m_1,n_1;2(r+1)),$ hence the girth is less than $2(r+1),$ it is even, so it must be $2r.$

The proof is complete.
\end{proof}

Notice that if $m_1=m$ and $n_1=n$ we obtain  \cite[Theorem 6]{AJRS22}.

\begin{theorem}\label{last-theorem}
Suppose that there exists a generalized $r$-gon of order $(m,n)$. Then 
there exist bipartite biregular graphs with parameters 
$(m,n+1;2r)$ and $(m+1,n;2r)$ of order 
$$n^{\frac{r}{2}-1}m^{\frac{r}{2}-1}( m +n+1).$$

In particular, if $m=n$ we obtain graphs of order $n^{r-2}(2n+1).$
\end{theorem}

\begin{proof}
We construct a bipartite biregular graph with parameters $(m,n+1;2r).$ The construction of the other graph is similar. 

In this proof, we also start with the Moore tree of a $(m+1,n+1;2r)$-bb cage $\Gamma$ with $\mathrm{deg}(v)=m$ and $\mathrm{deg}(v)=n$. Moreover, let $D_1(v)=\{ v_1,v_2,\dots ,v_m\}.$ 
Notice that, each vertex of $D_{r-1}^r(v,u)$ has exactly one neighbour in $D_{r-2}^{r-1}(v,v_1),$ because
$\Gamma $ is the Levi graph of a generalized polygon. 

Let
$$F=\bigcup _{j=1}^3  D_{r+1-j}^{r-j}(u,v)\bigcup _{i=2}^m D_{r-2}^{r-1}(v,v_i)\bigcup _{i=2}^m D_{r-3}^{r-2}(v,v_i)$$
denote a subset of vertices of ${\cal{T}}_{uv}.$
We construct a graph $\Gamma _F$ by deleting the edges and vertices of 
${\cal{T}}_{uv}$ other than the vertices in $F$ 
and the edges connecting them. We claim that $\Gamma _F$ is an $(m,n+1;2r)$-bb graph. It is sufficient to note that all vertices of $\Gamma$ being in the set 
$$ D_{r}^{r-1}(u,v)\cup D_{r-2}^{r-3}(u,v)\bigcup _{i=2}^m D_{r-3}^{r-2}(v,w_i),$$
have degree $m$ in $\Gamma _F$, whereas all vertices of $\Gamma$ being in the set 
$$ D_{r-1}^{r-2}(u,v)\bigcup _{i=2}^m D_{r-2}^{r-1}(v,w_i)$$
have yet degree $n+1$. 

Moreover, as we remove vertices and edges, the girth of the new graph is at least $2r$ and, by the lower bound given also in \cite{FilRamRivJaj19}\ for $(m,n+1;2(r+1))-bb\ graphs$ it does not have girth equal to $2(r+1)$. Consequently, the graph has girth exactly $2r$.

Finally, the order of $\Gamma _F$ is equal to \\
$n^{\frac{r}{2}-1}m^{\frac{r}{2}-2} +n^{\frac{r}{2}-1}m^{\frac{r}{2}-1}+n^{\frac{r}{2}}m^{\frac{r}{2}-1}+(m-1)[m^{\frac{r}{2}-2}n^{\frac{r}{2}-1}+m^{\frac{r}{2}-1}n^{\frac{r}{2}-1}] = \\
n^{\frac{r}{2}-1}m^{\frac{r}{2}-1}( m +n+1).$


\end{proof}

For $r=4$ the conditions of Theorem \ref{generalizationtoptree} can be simplified.
\begin{theorem}\label{generalizationtop_r=4}
Let $4\leq m\leq n$ be integers and suppose that there exists a generalized quadrangle 
$\mathcal{G}$ of order $(m,n)$. 
Let $m_1\leq n_1$ be integers satisfying the inequalities $m_1\leq m$ and $n_1\leq n$ 
Then 
there exist $(m_1,n_1;8)$-bb graphs of order
$(m_1+n_1)mn.$
\end{theorem}

\begin{proof}
Let $A,\, B,\, C,\, D$ be the vertices of a proper quadrangle $\mathcal{P}$ in 
$\mathcal{G}$. Then there are 
$$2\left( (1+m(n+1))+(1+m(n-1))\right) =4(mn+1)$$
points in $\mathcal{G}$ which are collinear with at least one vertex of $\mathcal{P}$.
As $4(mn+1)<(m+1)(mn+1)$,
there exists a pont $E$ which is not collinear with any vertex of $\mathcal{P}$.
The number of lines through $E$ is $n+1>4,$ hence we can choose a line $e$
through $E$ such that $e$ does not intersect any side of  $\mathcal{P}$.

Now we can repeat the proof of Theorem \ref{generalizationtoptree} in such a way that instead of an arbitrary edge $uv$ of the Moore tree, we choose the edge that corresponds to the incident point-line pair $(E,e)$.

\end{proof}
Table 1 contains 
the parameters of the known finite generalized $2r$-gons of order $(m,n)$ with $1<m\leq n$ and $r\geq 4$, the order of the graphs constructed in Theorem 
\ref{last-theorem}, the corresponding bipartite Moore bound and the excess of the
graph ($q$ denotes an arbitrary prime power).

\begin{center}
\begin{tabular}{|ccc|cc|c|}
\hline
 $m+1$	&   $n+1$	 & $2r$ &  $\frac{|F|}{m+n+1}$  & 
$\frac{B(m,n+1;2r)}{m+n+1}$ & excess \\
\hline
\hline
$q+1$ & $q+1$ & 8 & $q^2$ &  $q^2-q+1$ & $(2q+1)(q-1)$ \\ 
$q+1$ & $q^2+1$ & 8 & $q^3$ &  $q^3-q^2+1$ & $(q^2+q+1)(q^2-1)$ \\  
$q^2+1$ & $q^3+1$ & 8 & $q^5$ &  $q^5-q^3+1$ & $(q^3+q^2+1)(q^3-1)$ \\ 
$q$ & $q+2$ & 8 & $q^2-1$ &  $q^2-q$ & $(2q+1)(q-1)$ \\ 
$q+1$ & $q+1$ & 12 & $q^4$ &  $\frac{(q^2-q)^3-1}{q^2-q-1}$ & $(2q+1)(q-1)$ \\ 
$q+1$ & $q^3+1$ & 12 & $q^8$ &  $\frac{(q^4-q^3)^3-1}{q^4-q^3-1}$ & 
$\approx 2q^{10}$ \\
$q+1$ & $q^2+1$ & 16 & $q^9$ & $\frac{(q^3-q^2)^4-1}{q^3-q^2-1}$ &  
$\approx 3q^{11}$ \\  
\hline
\end{tabular}
\end{center}


Using the generalized quadrangle $T_2(O)$ we can make the previous constructions somewhat more explicit (for $r=4)$. We just use part of $T_2(O)$, namely we take a prime (slightly) larger than $n_1$. As in the construction of $T_2(O)$, we consider $\Pi =\mathrm{PG} (3,q)$ with its ideal plane $\Sigma$ and $O$ is an oval in $\Sigma$. Pick a line $\ell$ in $\Sigma $ disjoint from $O$. Finally, take $n_1$ points of $O$, and $m_1$ planes through $\ell$ in the affine part of $\Pi$. With this preparation, the construction is the following. (Note that $p+1$ plays the role of $m$ and $n$ in the previous theorems.)

\begin{proposition}
We define a bipartite graph $\Gamma$ with $V_1$ being the set of points on the union of the $m$ planes through $\ell$ and $V_2$ being the (affine) lines going to the $n$ points of $O$. Join a vertex of $V_1$ with a vertex of $V_2$ if the corresponding point and line are incident. Then $\Gamma$ is a bipartite biregular graph with degrees $m_1$ and $n_1$, of girth 8,  and $|V_1|=m_1p^2$ and $|V_2|=n_1p^2$.
\end{proposition}
\begin{proof}
The parallel affine planes through $\ell$ contain $p^2$ points each, so $|V_1|=m_1p^2$. Since there are $p^2$ affine lines through an ideal point, we also get $|V_2|=n_1p^2$. A line through a point of $O$ meets the union of $m_1$ parallel lines in exactly $m_1$ points, since $\ell$ and $O$ are disjoint. Conversely, one can join any point in the union of $m_1$ planes with any of the $n_1$ points of $O$. Hence the vertices of $V_1$ have degree $n_1$, and the vertices of $V_2$ have degree $m_1$. The girth is at least 8, since $\Gamma$ is a 
subgraph of the incidence graph of $T_2(O)$. If the girth was at least 10, then the number of vertices in the smaller bipartition class would be at least 
$1+n_1(m_1-1)+n_1(n_1-1)(m_1-1)^2$, which is larger than $m_1p^2$ Here we use that $p$ is close 
enough to $n$, the best results on how close $p$ can be to $n$ is due to Baker, Harman and Pintz \cite{BHP}. It shows that there is a prime $p$ between $x$ and $x+x^{0.525}$. 
\end{proof}

{\rm 
\begin{remark}
{\rm In the proposition above, $O$ was an oval and we chose some points of $O$. Actually, it is only needed that the set of  ideal points through which we consider the affine lines as $V_2$, form an arc (no three of them are on a line). Then it is easy to verify that the corresponding bipartite graph cannot have a 6-cycle, hence its girth is at least 8. }
\end{remark}

\begin{remark}
{\rm If $m_1$ is constant times $n_1$, then the above construction gives the correct order of magnitude (disregarding a multiplicative constant) for the number of vertices. If $m_1=c\sqrt n_1$ for some $c\le 1$, then one can copy the above construction for the generalized quadrangle $T_3(O)$ and again get the right order of magnitude (up to constant factor) for the number of vertices.  In this case $\sqrt n_1<p$, and the two classes have $m_1p^3$ and $n_1p^3$ vertices, respectively. The same remark is also true for Theorems \ref{generalizationtoptree}, \ref{generalizationtop_r=4}, \ref{last-theorem}, we have to choose a prime $p$ not much larger than $m$, the smaller parameter of the generalized $r$-gon. As far as we could check, the best result for the order of magnitude for such a prime $p$, is the one in \cite{BHP}. We also mention that instead of a prime $p$ it would be enough to find a prime-power $q$.} 
\end{remark}

\begin{remark}
{\rm The construction can also be modified to the girth 6 case. We choose a prime with $n_1<p$ and pick $n_1$ ideal points different from the one of horizontal lines. We also take $m_1$ horizontal lines. The bipartite graph will have $V_1$ as the set of $m_1p$ points on the $m$ distinguished horizontal lines and $V_2$ as the set of $n_1p$ (affine) line through the $n_1$ distinguished ideal points. If both $m_1$ and $n_1$ go to infinity, then this will be asymptotically the correct number of vertices. Again, this easy construction can be modified using projective spaces if $m_1$ and $n_1$ do not have the same order of magnitude.} 
\end{remark}

\subsection{Constructions by deleting substructures of generalized polygons}

In this subsection we will consider constructions where we delete special substructures of  generalized polygons as ovoids, spreads, sub-polygons, etc.

We start by considering the case when the girth is $g=2r$ and $r=4.$ In this case, if we take a generalized quadrangle $\mathcal{G}$ of order $(m,n)$ that contains an ovoid $\mathcal{O}$ (a set of $mn+1$ points, no two of them are collinear) or a spread $\mathcal{S}$ (a set of $mn+1$ lines, no two of them intersecting) and we delete the points of $\mathcal{O}$ or the lines of $\mathcal{S}$, then the Levi graph 
of the remaining structure is a  
bipartite biregular graph with parameters $(m,n+1;8)$ or $(m+1,n;8)$ and its order is
$$(m+n+2)(mn+1)-(mn+1)=(mn+1)(m+n+1).$$
Theorem \ref{last-theorem} results in $(m,n+1;8)$-graphs on  
$mn(m+n+1)$ vertices, hence those graphs have fewer vertices.
For some values of $m$ and $n$ even smaller graphs can be constructed by careful deletion of subquadrangles of generalized quadrangles.
\\
\begin{theorem}\label{deleted-subgq}
Suppose that there exists a generalized quadrangle $\mathcal{G}$ of order $(m,n)$ which contains a subquadrangle $\mathcal{G}'$ of order $(m,n/m)$.  Then there exists an
$(m,n+1;8)$-bb-graph of order $(m+n+1)\frac{m^2-1}{m}n.$ 
\end{theorem}

\begin{proof}
Delete the points and lines of $\mathcal{G}'$ from $\mathcal{G}$ and let $\Gamma $ 
be the Levi graph of the remaining structure. Then the girth of $\Gamma $ is $8,$
it has 
$$(m+1)(mn+1)-(m+1)\left( m\frac{n}{m}+1\right) =(m^2-1)n$$
points and 
$$(n+1)(mn+1)-\left( \frac{n}{m}+1\right) \left(m \frac{n}{m}+1\right)
=(m^2-1)n\frac{n+1}{m}$$
lines. 

There are $(n+1)-(\frac{n}{m}+1)=\frac{(m-1)n}{m}$ lines through each 
deleted points. As
$$(m+1)\left( m\frac{n}{m}+1\right) \frac{(m-1)n}{m}=(m^2-1)n\frac{n+1}{m},$$
each remaining line contains exactly one deleted point. So $\Gamma $ is an 
$(m,n+1)$-bb-graph of order $(m+n+1)\frac{m^2-1}{m}n.$ 
\end{proof}





This construction works in the following cases:
\begin{enumerate}
\item
$\mathcal{Q}(4,q)$ contains $\mathcal{Q}(3,q)$ as a subquadrangle, in this case
$m=n=q$, the order of $\Gamma $ is $(q^2-1)(2q+1).$
\item
$\mathcal{Q}(5,q)$ contains $\mathcal{Q}(4,q)$ as a subquadrangle;
$T_3(\mathcal{O})$ contains $T_2(\mathcal{O}')$ as a subquadrangle, in these cases
$m=q$ and $n=q^2$, the order of $\Gamma $ is $(q^3-1)q(q+1).$
\item
$\mathcal{H}(4,q)$ contains $\mathcal{H}(3,q)$ as a subquadrangle, in this case
$m=q^2$ and $n=q^3$, the order of $\Gamma $ is $(q^4-1)q(q^3+q^2+1).$
\end{enumerate}

These GQs belong to the class of projective generalized polygons. 
A finite generalized polygon $\mathcal{G}=(\mathcal{P},\mathcal{L},\mathrm{I})$ is called \emph{projective}, if there exists a projective space 
$\mathcal{S}=\mathrm{PG}(d,q)$ for which  $\mathcal{P}$ is a subset of the set of points of $\mathcal{S}$, $\mathcal{L}$ is a subset of the set of lines of $\mathcal{S}$ and the incidence $\mathrm{I}$ is the one inherited from $\mathcal{S}.$ The space $\mathcal{S}$
is called the \emph{ambient space} of $\mathcal{G}$ if no proper subspace of it contains $\mathcal{P}.$
For this class a more geometric construction gives better graphs in some cases. 

\begin{theorem}\label{proj-torl}
Let $r>2$ be even number and suppose that there exists a finite projective generalized 
$r$-gon $\mathcal{G}$ of order $(m,n)$ whose ambient space is $\mathrm{PG}(d,q).$ Let $\mathcal{H}$ be a hyperplane in $\mathrm{PG}(d,q)$ which contains $u$ points of $\mathcal{G}.$ Then there exist an $(m,n+1;2r)$-bb-graph of order 
$$  
(m+n+1)\left( \frac{(m+1)((mn)^{\frac{r}{2}}-1)}{m(mn-1)}-\frac{u}{m}\right) .
$$
\end{theorem}

\begin{proof}
Let $v$ denote the number of  those lines of 
$\mathcal{G}$ which are entirely contained in $\mathcal{H}$. Delete these lines and
the $u$ points of $\mathcal{G}$ which are contained in $\mathcal{H}.$ Each remained line intersects $\mathcal{H}$ in exactly 1 point. Hence the double counting of the incident
remained point - remained line pairs gives 
$$(n+1)\left((m+1) \frac{((mn)^{\frac{r}{2}}-1)}{mn-1} -u\right) =
m\left((n+1) \frac{((mn)^{\frac{r}{2}}-1)}{mn-1} -v\right) ,$$
so the number of remained lines equals to $\frac{n+1}{m}$ times the number of 
remained points.
$$u+v=\frac{m+n+1}{m}u-\frac{(n+1)((mn)^{\frac{r}{2}}-1)}{m(mn-1)} .$$
Thus the Levi graph $\Gamma $ of the remained structure is an 
$(m,n+1;2r)$-bb-graph on
$$
\begin{aligned}
 & \left(1+\frac{n+1}{m}\right) \left((m+1) \frac{(mn)^{\frac{r}{2}}-1)}{mn-1} -u\right) = \\
 & (m+n+1)\left( \frac{(m+1)((mn)^{\frac{r}{2}}-1)}{m(mn-1)}-\frac{u}{m}\right) 
\end{aligned}
$$
vertices.
\end{proof}

The constructions presented in Theorems \ref{deleted-subgq} and \ref{proj-torl}
often result the same graphs. 

Combining Corollary \ref{cor-8} and non-existence results for generalized polygons,
it is easy to improve the bipartite Moore bound in some cases.

\begin{theorem}
\label{betterbound}
Let $3\leq m$ be an integer. Then
\begin{itemize}
\item[{(i)}]
$B_c(m,m+1;8)\geq (2m+1)(m^2-m+2)$;
\item[{(ii)}]
$B_c(m,m^2+1;8)\geq (m^2+m+1)(m^3-m^2+2)$;
\item[{(iii)}]
$B_c(m,m+1;12)\geq (2m+1)(m^4-2m^3+2m^2-m+2)$.
\end{itemize}
\end{theorem}

\begin{proof} $(i)$: It was proved by Feit and Higman \cite{FH} (see also \cite[Theorem 10.23.]{KSz} or \cite[1.2.2.]{PT}) that if there exists a thick generalized quadrangle of order $(s,t),$ then $s+t$ divides $st(s+1)(t+1).$  If $m>2$, then $2m-1$ does not divide 
$(m-1)m^2(m+1),$ hence there is no generalized quadrangle of order $(m-1,m)$. 
As $m$ and $m+1$ are coprime, Corollary \ref{cor-8} gives that
$$
\begin{aligned}
B_c(m,m+1;8) & \geq B(m,m+1;8)+m+(m+1) \\ 
& =(2m+1)(m^2-m+2).
\end{aligned}
$$

$(ii)$: By Higman's inequality \cite{H75} (see also \cite[Theorem 10.24.]{KSz} or \cite[1.2.3.]{PT}), if 
$1<m\leq n$ and there exists a thick generalized quadrangle of order $(s,t),$ then 
$s^2\geq t.$  Hence, for $m>2$ there is no generalized quadrangle of order 
$(m-1,m^2).$  
As $m$ and $m^2+1$ are coprime, Corollary \ref{cor-8} gives that
$$
\begin{aligned}
B_c(m,m^2+1;8) & \geq B(m,m^2+1;8)+m+(m^2+1) \\ 
& =(m^2+m+1)(m^3-m^2+2) 
\end{aligned}
$$

$(iii)$:  It was proved by Feit and Higman \cite{FH} (see also 
\cite[Theorem 1.7.1.(iv)]{VM98}) that if $s,t>1$ and there exists a generalized hexagon of order $(s,t),$ then $st$ is a 
perfect square. If $m>2,$ then $(m-1)m$ is not a perfect square, hence there is no generalized hexagon of order $(m-1,m)$. 
As $m$ and $m+1$ are coprime, Corollary \ref{cor-8} gives that 
$$
\begin{aligned}
B_c(m,m+1;12) & \geq  B(m,m+1;12)+m+(m+1) \\ 
& =(2m+1)(m^4-2m^3+2m^2-m+2).
\end{aligned}
$$
\end{proof}

\begin{theorem}\label{q-(q+1)}
If $q$ is a prime power, then there exist a $(q,q+1;8)$-bb graph of order $(2q+1)(q^2-1).$ 

In particular, for $q=3$ there exists a $(3,4;8)$-bb-cage on $56$ vertices.
\end{theorem}

\begin{proof}
Let $\mathcal{P}$ be a parabolic quadric in PG(4,q) and $\mathcal{H}$ be a hyperplane such that
$\mathcal{P}\cap \mathcal{H}$ is a hyperbolic quadric (e.g: if $\mathcal{P}$ has equation $X_0^2+X_1X_2+X_3X_4=0,$
then the equation of $\mathcal{H}$ could be $X_0=0$). The points and lines entirely contained in $\mathcal{P}$ form the
classical generalized quadrangle $\mathcal{Q}(4,q)$, whose order is $(q,q)$,
so it has $q^3+q^2+q+1$ points and the same number of lines. Delete the $2(q+1)$ lines and $(q+1)^2$ points of $\mathcal{P}\cap \mathcal{H}$. By Theorem
\ref{proj-torl}, the Levi graph $\Gamma $ of the remaining structure is a 
$(q,q+1;8)$-bb-graph of order 
$$\left( (q^3+q^2+q+1)-2(q+1)\right)   + \left( (q^3+q^2+q+1)
-(q+1)^2\right) =(2q+1)(q^2-1).$$

By Theorem \ref{betterbound}$(i)$, $B_c(q,q+1;8)=(2q+1)(q^2-q+2).$ So 
for $q>2$ the graph $\G $ has $B_c(q,q+1;8)+(2q+1)(q-3)$ vertices. 
In particular, for $q=3$ we get a $(3,4;8)$-bb-cage on 56 vertices.
\end{proof}



\begin{theorem}\label{q-(q+1)}
If $q$ is a prime power, then there exist a $(q,q^2+1;8)$-bb-graph of order $(q^2+q+1)(q^3-q).$ 
\end{theorem}

\begin{proof}
Let $\mathcal{E}$ be an elliptic quadric in PG(5,q) and $\mathcal{H}$ be a 
non-tangent hyperplane of $\mathcal{E}.$ Then
$\mathcal{E}\cap \mathcal{H}$ is a parabolic quadric. The points and lines entirely contained in $\mathcal{E}$ form the
classical generalized quadrangle $\mathcal{Q}(5,q)$, whose order is $(q,q^2)$.
Delete the lines and points of $\mathcal{P}\cap \mathcal{H}$. By Theorem
\ref{proj-torl}, the Levi graph $\Gamma $ of the remaining structure is a 
$(q,q^2+1;8)$-bb-graph of order $(q^2+q+1)(q^3-q)$.
\end{proof}

\noindent 
\begin{remark}
{\rm 
By Theorem \ref{betterbound}$(ii)$, $B_c(q,q^2+1;8)=(q^2+q+1)(q^3-q^2+2)$. 
So the graph $\G $ has $B_c(q,q^2+1;8)+(q^2+q+1)(q^2-q-2)$ vertices.}
\end{remark} 
%

The method of Theorem \ref{q-(q+1)} can be extended to a class of generalized 
hexagons, too. For the detailed description of  generalized hexagons we refer the
reader to the book of Van Maldeghem \cite{VM98}, a brief description can also be found in 
\cite[Example 10.39]{KSz}.

\begin{theorem}\label{gen-hex}
If $q$ is a prime power, then there exist a $(q,q+1;12)$-bb-graph
of order $(2q+1)(q^4-q).$
\end{theorem}

\begin{proof}
Let $\mathcal{P}$ be a parabolic quadric in PG(6,q) and $\mathcal{H}$ be a hyperplane such that
$\mathcal{P}\cap \mathcal{H}$ is a hyperbolic quadric (e.g: if $\mathcal{P}$ has equation $X_0^2+X_1X_2+X_3X_4+X_5X_6=0,$
then the equation of $\mathcal{H}$ could be $X_0=0$). The points and a subset 
$\mathcal{L}$
of the lines entirely contained in $\mathcal{P}$ form the split Cayley hexagon $H_q$, whose order is $(q,q)$.  
Delete the 
points of $\mathcal{P}\cap \mathcal{H}$ and those lines of 
$\mathcal{L}$ which are entirely contained in $\mathcal{H}$.


Then, by Theorem \ref{proj-torl}, the Levi graph $\Gamma $ of the remaining structure is a $(q,q+1;12)$-bb-graph on
$$(q^5-q^2) + (q^4-q)(q+1) =(2q+1)(q^4-q)$$
vertices.
\end{proof}

\noindent 
\begin{remark}
{\rm By Theorem \ref{betterbound}$(iii)$, $B_c(q,q+1;12)=(2q+1)(q^4-2q^3+2q^2-q+2).$ 
So the graph $\G $ has $B_c(q,q+1;12)+(2q+1)(2q^3-2q^2-2)$ vertices. 
}
%
\end{remark}

\begin{remark}
{\rm For $r=6$ and $m=n=q$ Theorem \ref{last-theorem} gives a graph on $(2q+1)q^4$
vertices.}
\end{remark}

\noindent
{\bf Gabriela Araujo-Pardo:} 
Instituto de Matem\'aticas-Campus Juriquilla, Universidad Nacional Aut\'onoma de 
M\'exico, C.P. 076230, Boulevard Juriquilla \# 3001, Juriquilla, Qro., M\'exico;  \\
 e-mail:  \texttt{garaujo@im.unam.mx} \\

\noindent
{\bf Gy\"orgy Kiss:} Department of Geometry and ELKH-ELTE Geometric and Algebraic Combinatorics
Research Group, E\"otv\"os Lor\'and University, 1117 Budapest, P\'azm\'any
s. 1/c, Hungary; and Faculty of Mathematics, Natural Sciences and Information Technologies, University of Primorska, Glagolja\v ska 8, 6000 Koper,
Slovenia; \\
e-mail: \texttt{gyorgy.kiss@ttk.elte.hu} \\

\noindent
{\bf Tam\'as Sz\H{o}nyi:} Department of Computer Science and ELKH-ELTE Geometric and Algebraic Combinatorics
Research Group, E\"otv\"os Lor\'and University, 1117 Budapest, P\'azm\'any
s. 1/c, Hungary; and Faculty of Mathematics, Natural Sciences and Information Technologies, University of Primorska, Glagolja\v ska 8, 6000 Koper,
Slovenia; \\ e-mail: \texttt{tamas.szonyi@ttk.elte.hu}

\end{document}